\documentclass[a4paper, 10pt, conference]{ieeeconf}      

\IEEEoverridecommandlockouts                              

\usepackage[noadjust]{cite}  
\usepackage{amsmath,amssymb}
\usepackage{aligned-overset}
\usepackage{enumerate}
\usepackage{bm}
\usepackage{amsfonts}
\usepackage{cancel}

\usepackage{soul}

\newtheorem{theorem}{Theorem}[section]
\newtheorem{lemma}{Lemma}[section]

\newtheorem{remark}{Remark}[section]
\newtheorem{assumption}{Assumption}[section]
\newtheorem{problem}{Problem}[section]

\newcommand\numberthis{\addtocounter{equation}{1}\tag{\theequation}}
\makeatletter

\usepackage{algpseudocode}
\usepackage{algorithm}

\usepackage{setspace}
    
\usepackage{bm} 

\allowdisplaybreaks  

\usepackage[unicode]{hyperref}
\hypersetup{
    colorlinks,
    citecolor=black,
    filecolor=black,
    linkcolor=black,
    urlcolor=black
}

\usepackage{nccmath} 

\usepackage[font=sf]{caption}
\usepackage{accents}

\usepackage{nccmath} 

\title{\LARGE \bf Communication-Free Bearing-Only Distributed Target Localization and Circumnavigation by Multiple Agents}

\author{Donglin Sui$^{1}$ and Mohammad Deghat$^{1}$
\thanks{$^{1}$The authors are with the School of Mechanical and Manufacturing Engineering, University of New South Wales, Sydney, Australia (e-mail: {\tt\small d.sui@unsw.edu.au, m.deghat@unsw.edu.au}).}%
}

\begin{document}

\maketitle
\thispagestyle{empty}
\pagestyle{empty}

\begin{abstract}
    This paper considers the problem of localization and circumnavigation of an unknown stationary target by a group of autonomous agents using only local bearing measurements. We assume that no direct communication or exchange of information is permitted among the group of agents and propose an algorithm to approximate the angular separation between agents and a distributed control law that forces the agents to circumnavigate the target at a desired orbit with any prescribed angular separation. Global asymptotic stability of the system is analyzed rigorously using Lyapunov theory, cascade control strategy, and perturbation method. The performance of the proposed algorithm is verified through simulations.
\end{abstract}


\section{INTRODUCTION}
\label{sec:Introduction}
    The task of target circumnavigation has drawn extensive research attention in recent decades owing to its wide applications in both civil and military fields. 
    Circumnavigation tasks are typically performed by steering one or more agents toward and then on a circular trajectory centered at the target with a prescribed radius. As the target to be monitored in practical applications is often uncooperative such that its location is unknown to the agent(s), the circumnavigation task usually encompasses two sub-problems, namely, the target localization and the formation maneuver around the target. The literature often refers to this kind of uncooperative circumnavigation task as a dual control problem \cite{feldbaum1963Dual}. 
    
    In existing studies, the assumption that the sensing capability of the agent(s) provides information-rich measurements (such as distance, bearing angle, and velocity) of the target is frequently adopted to decouple the two sub-problems and thereby simplify the circumnavigation task; see for example \cite{kim2007Cooperative,shames2011Close,wang2017LimitCycleBased,zhang2019Coordinated,sen2021Circumnavigation,dou2021MovingTarget}. However, the circumnavigation tasks are often subject to measurement incompleteness due to practical circumstances such as limited payload, and therefore, it is preferred to restrict the available sensor knowledge in theoretical development. When the agent is required to remain in stealth or radio silence, bearing-only measurement is necessitated as it is a passive measurement technique \cite{deghat2010Target}. Due to this merit, extensive research efforts \cite{deghat2010Target,deghat2014Localization,qi2019Virtual,yu2019Bearing,chen2022Multicircular,chen2023Finitetime} have been devoted to the problem of bearing-only target localization and circumnavigation (abbrev. BoTLC hereafter).
    
    In comparison to single-agent systems, multi-agent systems are superior in meeting the increasingly complex and diverse circumnavigation task due to their enhanced robustness and versatility. A typical approach to coordinate the group of agents in BoTLC is to distribute the agents uniformly or distribute them according to a set of pre-defined separation angles on a single or multiple circles centered at the target, where the separation angle refers to the angle subtended at the target by two neighbor agents, see for instance \cite{swartling2014Collective,boccia2017Tracking,dou2019Distributed,chen2022Target,chen2022Multicircular,chen2023Finitetime,ma2023Finitetime}. Along a different train of thought, controllers used in \cite{zheng2015Enclosing,xiong2022Cooperative} drive the agents to form an evenly distributed circular formation using only a ratio of bearing angles subtended at an agent by its neighbor agents and the target, whereas \cite{yu2018Distributed,zhu2020Enclose,zou2022CoordinateFree} make use of heading angles of neighbor agents to achieve formation maneuver. In a different vein, authors in \cite{yu2019Bearing} employed the bearing rigidity theory to tackle the circumnavigation task. 
    
    However, with the exception of \cite{yu2018Distributed,zheng2015Enclosing,xiong2022Cooperative}, all studies cited above assume there exists at least some local communication, to deliver necessary information, among agents. The reliance on inter-agent communication can be unfavorable in applications where the system is exposed to adversarial attacks targeting radio channels such as jamming or spoofing. The need for communication also significantly compromises the stealthiness of the system, which is one of the driving force to study the problem of BoTLC. Further, a communication-free system offers many appealing features such as not suffering from transmission errors such as latency, data collision, and packet loss, and an enhanced capability to be deployed in harsh environments such as underwater. Therefore, the aforementioned discussions prompt this paper to study the BoTLC problem under the assumption that no communication is permitted among agents. 
    
    \textit{Contributions.} A novel communication-free algorithm is proposed in this paper for arbitrarily spaced circular formation on the problem of multi-agent BoTLC. The paper's primary contributions are threefold: Firstly, to the best of the authors' knowledge, the proposed method represents the first result on a fully communication-free circumnavigation algorithm for a multi-agent system using solely bearing measurements. Avoiding communications among agents has significantly complicated the stability analysis and offers evident practical applications. In comparison to \cite{yu2019Bearing,swartling2014Collective,boccia2017Tracking,dou2019Distributed,chen2022Target,chen2022Multicircular,chen2023Finitetime,ma2023Finitetime}, the proposed controller does not require any inter-agent communication and uses strictly local on-board bearing measurements. In contrast to \cite{zheng2015Enclosing,xiong2022Cooperative}, global rather than local asymptotic stability is rigorously derived in this paper, and furthermore, the angular separation between agents can be arbitrarily assigned. Unlike \cite{yu2018Distributed,zhu2020Enclose,zou2022CoordinateFree}, the proposed controller does not need to use additional information of the heading angle of neighboring agents. Secondly, all measurements are within local coordinate frames that are not necessarily aligned. Last but not least, unlike most of the existing studies, agents are not required to possess any knowledge about the total number of agents in the formation.
    

    The rest of this paper is structured as follows. Section~\ref{sec:Preliminaries} formally states the BoTLC problem. The proposed solution is formulated in Section~\ref{sec:Proposed_algorithm}. Section~\ref{sec:Stability_proofs} is devoted to the stability analysis. Section~\ref{sec:Simulation_results} shows the simulation results, and Section~\ref{sec:Conclusion} provides concluding remarks.


\section{PRELIMINARIES AND PROBLEM STATEMENT}
\label{sec:Preliminaries}

    \subsection{Graph Theory}
    \label{sec:Graph_Theory}
        The interaction topology of the multi-agent system is described by a directed graph (digraph) $\mathcal{G}=(\mathcal{V},\mathcal{E},\mathcal{A})$. The group of agents are represented by the vertex set $\mathcal{V} = \{1,2,\dots,n\}$. $\mathcal{E} = \{e_{ij} = (i,j)\} \subseteq \mathcal{V} \times \mathcal{V}$ is the edge set where the edge $e_{ij} = (i,j) \in \mathcal{E}$ indicates that agent $i$ considers agent $j$ as its neighbor. The neighbor strategy will be explained later. The adjacency matrix is denoted by $\mathcal{A} = \left[ a_{ij} \right] \in \mathbb{R}^{n\times n}$ where $a_{ij}=1$ if $e_{ij}\in \mathcal{E}$ and $a_{ij}=0$ otherwise. We stipulate that $\mathcal{G}$ contains no self-loops, that is, $a_{ii}=0$. Let $\mathcal{N}_{i} = \left\{ j : e_{ij} \in \mathcal{E} \right\}$ denote the set of agent $i$'s neighbors. The Laplacian matrix of the digraph $\mathcal{G}$ is denoted by $\mathcal{L}(\mathcal{G}) = \left[\ell_{ij}\right]\in \mathbb{R}^{n\times n}$ where $\ell_{ij}=-a_{ij}$ for $i\neq j$ and $\ell_{ii} = \sum_{j=1}^{n} a_{ij}$.

    \subsection{Useful Lemmas}
    \label{sec:Useful_lemmas}
        \begin{lemma}[{\cite[Theorem~7\&8]{olfati-saber2004Consensus}}]
            \label{lemma:algebraic_connectivity}
            \normalfont If a digraph denoted by $\mathcal{G}$ is strongly connected, balanced, and nonswitching, then the generalized Fiedler's notion of algebraic connectivity $\lambda_2(\mathcal{G})$ has the following property on $\mathcal{G}$:
            \begin{equation}
                \label{eq:algebraic_connectivity}
                \inf_{\substack{\bm{x}\neq 0\\ \bm{1}_{n}^{\top}\bm{x}=0}} \medmath{\frac{\bm{x}^\top \mathcal{L}(\mathcal{G})\bm{x}}{\bm{x}^{\top}\bm{x}}} = \lambda_2(\mathcal{G}),
            \end{equation}
        \end{lemma}
        where $\bm{x}=[x_1,\dots,x_n]^\top \in \mathbb{R}^n$, and $0=\lambda_1 < \lambda_2(\mathcal{G})\leq \dots \leq \lambda_n$ are eigenvalues of the Laplacian matrix $\mathcal{L}(\mathcal{G})$. 
        
        \begin{lemma}[\cite{sontag1989Remarks,loria2005Cascaded}]
            \label{lemma:cascaded system}
            Consider the cascaded nonlinear systems taking the form of
            \newsavebox{\mycases}
            \begin{subequations}
                \begin{align}
                    \sbox{\mycases}{$\displaystyle \Sigma\left\{\begin{array}{@{}c@{}}\vphantom{1\ x\geq0}\\\vphantom{0\ x<0}\end{array}\right.\kern-\nulldelimiterspace$}
                  \raisebox{-.5\ht\mycases}[0pt][0pt]{\usebox{\mycases}}\Sigma_1 \; : \; &\dot{\bm{x}}_1 = \bm{f}_1 ( \bm{x}_1 , \bm{x}_2 ), \label{eq:cascade_system_1} \\
                     \Sigma_2 \; : \; &\dot{\bm{x}}_2 = \bm{f}_2 ( \bm{x}_2 ). \label{eq:cascade_system_2}
                \end{align}
            \end{subequations}
            The origin of $\Sigma$ is globally asymptotically stable (GAS) if the following conditions are met:
            \begin{enumerate}[({C}1)]
                \item The origin of $\dot{\bm{x}}_1 = \bm{f}_1 (\bm{x}_1, \bm{0})$ is GAS;
                \item The origin of $\dot{\bm{x}}_2 = \bm{f}_2 (\bm{x}_2)$ is GAS;
                \item The solution of $\dot{\bm{x}}_1 = \bm{f}_1(\bm{x}_1, \bm{x}_2)$ is bounded in the sense of Converging Input - Bounded State (CIBS) \cite{sontag1989Remarks}, that is, for each input $\bm{x}_2 (t)$ such that $\lim\limits_{t\to\infty} \bm{x}_2(t) = \bm{0}$, and for each initial state $\bm{x}_1(0) = \bm{x}_{1,o}$, the solution of $\Sigma_1$ with $\bm{x}_{1,o}$ exists for all $t\geq 0$ and is bounded.
            \end{enumerate}
        \end{lemma}

        \begin{lemma}[{\cite[Theorem~4.9]{khalil2014Nonlinear}}]
            \label{lemma:uniform stability}
            \normalfont Consider the system 
            \begin{equation}
                \label{eq:uniform stability}
                \dot{\bm{x}} = \bm{f}(t,\bm{x}).
            \end{equation}
            Suppose $\bm{x}=\bm{0}$ is an equilibrium point of \eqref{eq:uniform stability} and $D \subset \mathbb{R}^n$ is a domain containing $\bm{x}=\bm{0}$. Let $W_1(\bm{x})$, $W_2(\bm{x})$, and $W_3(\bm{x})$ be continuous positive definite functions on $D$, and let $V:[0,\infty) \times D \to \mathbb{R}$ be a continuously differentiable function such that 
            \begin{align}
                W_1(\bm{x}) &\leq V(t,\bm{x}) \leq W_2 (\bm{x}), \\
                \medmath{\frac{\partial V}{\partial t}} &+ \medmath{\frac{\partial V}{\partial \bm{x}}} \bm{f}(t,\bm{x}) \leq -W_3(\bm{x}),
            \end{align}
            for all $t\geq 0$ and for all $\bm{x}\in D$. Then, $\bm{x}=\bm{0}$ is uniformly asymptotically stable. Moreover, if $r$ and $c$ are chosen such that $B_r = \{||\bm{x}|| \leq r \} \subset D$ and $c < \min_{||\bm{x}||=r} W_1(\bm{x})$, then every trajectory starting in $\{\bm{x}\in B_r | W_2(\bm{x} )\leq c\}$ satisfies 
            \begin{equation}
                ||\bm{x}(t) || \leq \beta(||\bm{x}(t_0)||, \, t-t_0), \quad \forall t \geq t_0 \geq 0,
            \end{equation}
            for some class $\mathcal{K}\mathcal{L}$ function $\beta$. Finally, if $D=\mathbb{R}^n$ and $W_1(\bm{x})$ is radially unbounded, then $\bm{x}=\bm{0}$ is  uniformly globally asymptotically stable (UGAS).
        \end{lemma}

        \begin{lemma}[{\cite[Theorem~9.1]{khalil2014Nonlinear}}] 
            \label{lemma:perturbed_sys}
            \normalfont Consider the nominal system $\Pi$ and the perturbed system $\Pi '$ given as follows,
            \begin{subequations}
                \begin{align}
                    \Pi \; : \; \dot{\bm{x}} &= \bm{q}(t,\bm{x}) ,  \label{eq:nominal sys} \\
                    \Pi ' \; : \; \dot{\bm{x}} &= \bm{q}(t,\bm{x}) + \bm{g}(t,\bm{x}) .  \label{eq:perturbed sys}
                \end{align}
            \end{subequations}
            Let $D$ be a domain that contains the origin, and suppose the following conditions are met:
            \begin{enumerate}[({C}1)]
                \item The perturbation term $\bm{g}(t,\bm{x})$ is piecewise continuous in $t$, locally Lipschitz in $\bm{x}$, and there exists a positive constant $\varpi$ such that
                \begin{equation}
                    \label{eq:bounded perturbation}
                    ||\bm{g}(t,\bm{x})|| \leq \varpi,\quad \forall (t,\bm{x}) \in [0,\infty) \times D;
                \end{equation}
                \item The origin of the nominal system \eqref{eq:nominal sys} is an exponentially stable equilibrium point;
                \item There exists a Lyapunov function $V(t,\bm{x})$ that satisfies the conditions of \textit{Lemma~\ref{lemma:uniform stability}} for the nominal system \eqref{eq:nominal sys} for $(t,\bm{x}) \in [0,\infty)\times D$ and $\{W_1(\bm{x}) \leq c\}$ is a compact subset of $D$.
            \end{enumerate}
            Let $\bm{y}(t)$ and $\bm{z}(t)$ denote solutions of the nominal system \eqref{eq:nominal sys} and the perturbed system \eqref{eq:perturbed sys}, respectively. Then, for each compact set $\Omega \subset \{W_2(\bm{x}) \leq \rho c, \, 0< \rho < 1\}$, there exist positive constants $\chi$, $\gamma$, $\mu$, and $k$, independent of $\varpi$, such that if $\bm{y}(t_0) \in \Omega$, $\delta < \eta$, and $||\bm{z}(t_0) - \bm{y}(t_0)|| < \mu$, the solutions $\bm{y}(t)$ and $\bm{z}(t)$ will be uniformly bounded for all $t\geq t_0 \geq 0$ and
            \begin{equation}
                ||\bm{z}(t) - \bm{y}(t)|| \leq k e^{-\gamma (t-t_0)} ||\bm{z}(t_0) - \bm{y}(t_0)|| + \chi \varpi .
            \end{equation}
        \end{lemma}

    \subsection{Background and Notations}
        Consider a 2D plane containing a group of $n\geq 2$ agents and a target to be localized and circumnavigated by the agents. The location of agent $i$, $i \in \mathcal{V}$, at time $t$ is denoted by $\bm{p}_i(t) = \left[x_i (t), y_i (t) \right]^\top \in \mathbb{R}^2$ and that of the stationary target by $\bm{x} = \left[x_t , y_t \right]^\top \in \mathbb{R}^2$.  The agents are viewed as holonomic moving points whose kinematics are described by the following single-integrator model,
        \begin{equation}
            \label{eq:agent_kinematics}
            \dot{\bm{p}}_i (t) = \bm{u}_i (t) , \; i \in \mathcal{V},
        \end{equation}
        where $\bm{u}_i(t)$ is the control input to be designed. 

        Agent $i$'s local bearing measurement of the target at time $t\geq 0$ is expressed in the form of a unit bearing vector $\bm{\varphi}_{iT}(t)$ starting from agent $i$ and pointing to the target, 
        \begin{equation}
            \label{eq:def_of_varphi_iT}
            \bm{\varphi}_{iT} (t) = \medmath{\frac{\bm{x}- \bm{p}_i (t)}{|| \bm{x} - \bm{p}_i (t) ||}} = \medmath{\frac{\bm{x} - \bm{p}_i (t)}{d_i (t)} }, \; i \in \mathcal{V},
        \end{equation} 
        where $d_i (t)$ is the Euclidean distance between agent $i$ and the target at time $t$. Similarly to \eqref{eq:def_of_varphi_iT}, let $\bm{\varphi}_{ij} (t) \in \mathbb{R}^2$ be the unit bearing vector on the line passing through $\bm{p}_i(t)$ and $\bm{p}_j(t)$, and $\psi_{ij}(t)\in[0,2\pi)$ be the angular spacing rotated counterclockwise from $\bm{\varphi}_{iT}(t)$ to $\bm{\varphi}_{ij}(t)$. Unit vector $\bar{\bm{\varphi}}_{iT} (t) \in \mathbb{R}^2$ is obtained by rotating $\bm{\varphi}_{iT}(t)$ clockwise by $\pi/2$. Angle $\beta_{ij}(t)\in [0,2\pi)$ represents the actual angular separation between agent $i$ and agent $j$ at time $t\geq 0$, which is not directly available to either agent $i$ or $j$. A special case of a two-agent system is depicted in Fig.~\ref{fig:notation_conventions}. Symbols $\bm{U}_{i}$, $\xi_{i}(t)$, $\gamma_{i}(t)$, and $\beta_{i}(t)$ will be explained in later sections.
        \begin{figure}[h]
            \centering
            \includegraphics[width=0.89\linewidth]{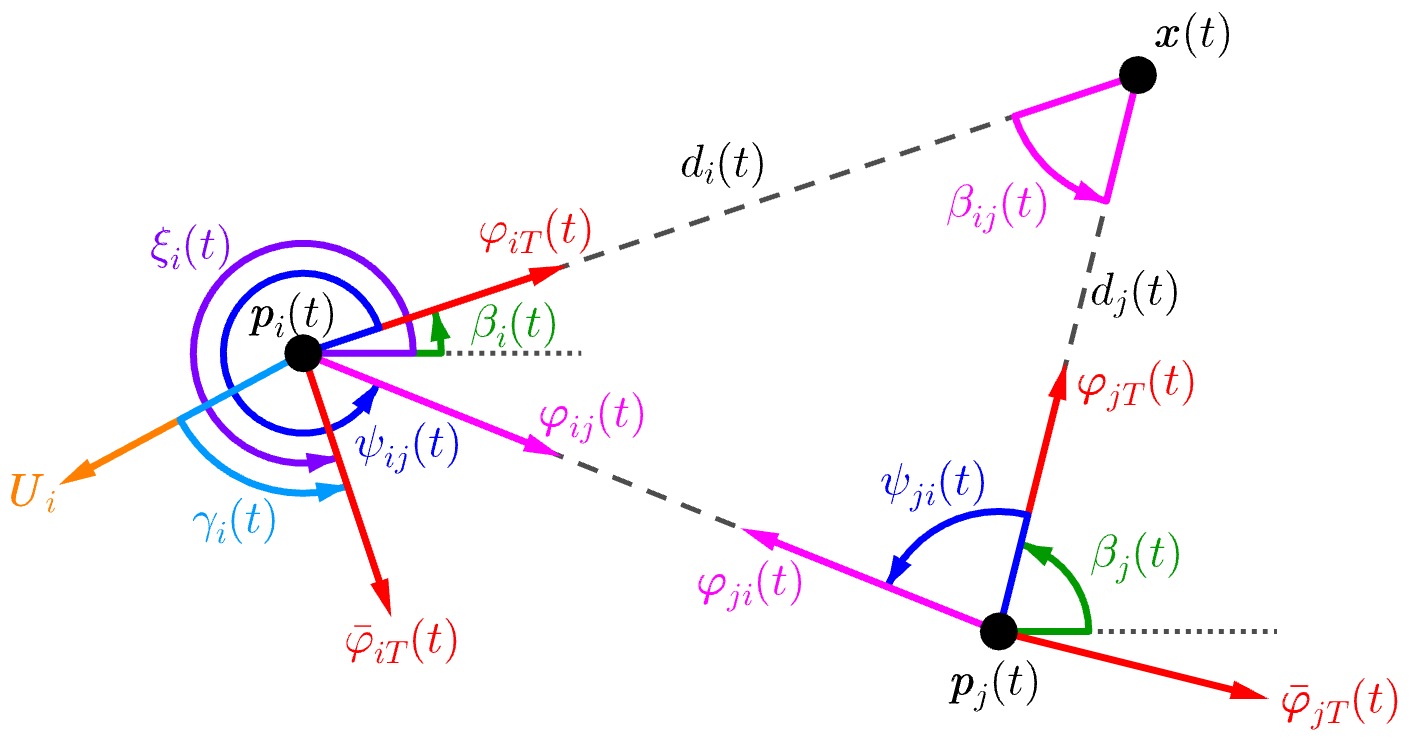}
            \caption{\label{fig:notation_conventions}A graphical illustration of notations.}
        \end{figure}
        
    
    \subsection{Problem Statements and Assumptions}
        In this paper, we aim to simultaneously solve the following two problems.
        \begin{problem}[Localization]
            \label{problem:Localization_Problem}
            \normalfont Design a target estimator $\dot{\hat{\bm{x}}}_i (t)$ for $i\in \mathcal{V}$, such that each agent localizes a stationary target using only local bearing measurements. In particular, for all $i\in\mathcal{V}$, the estimation error $\tilde{\bm{x}}_i (t)$, defined as $\tilde{\bm{x}}_i(t) = \hat{\bm{x}}_i(t) - \bm{x}$, converges to zero.
        \end{problem}
        
        \begin{problem}[Circumnavigation and Circular Formation]
            \label{problem:Circum_n_Formation_Problem}
            \normalfont Given a circular formation characterized by a desired circumnavigation distance $d^{*}>0$ and a sequence of desired angular separation $\bm{\beta}^* = \left[ \beta_{12}^{*}, \beta_{23}^{*}, \dots, \beta_{n-1,n}^{*},\beta_{n,1}^{*} \right]^{\top}$, design a distributed control law $\bm{u}_i(t)$ such that for any multi-agent system with $n\geq 2$, each agent converges to the common circle centered at the target with the desired radius $d^{*}$ while circumnavigating the target and maintaining the desired angular spacing $\bm{\beta}^{*}$. In particular, the proposed control law should satisfy the following conditions:
            \begin{enumerate}[(1)]
                \item The tracking error converges to zero:
                \begin{equation}
                    \label{eq:distance error condition}
                    \lim_{t\to \infty} || \bm{x}  - \bm{p}_i (t) || - d^{*} = 0, \quad \forall \, i \in \mathcal{V}.
                \end{equation}
                \item The angular separation error $\tilde{b}_{i}(t)$, defined as $\tilde{b}_{i}(t) = \beta_{ij}(t) - \beta_{ij}^{*}$, converges to zero for all $(i,j)\in\mathcal{E}$.
            \end{enumerate}
        \end{problem}
        To address the above problems, we make the following assumptions.
        \begin{assumption}[Initial Positions]
            \label{assumption:Initial_positions}
            We assume that all agents and the target are occupying different locations at $t=0$. As the problem of collision avoidance between agents is not a focus of the current paper, we further assume that the initial positions of the group of agents are almost in order \cite{zhao2016Bearing,wang2017LimitCycleBased} such that no inter-agent collision could have evolved from the considered initial conditions.
        \end{assumption}

        \begin{assumption}[Zero Communication]
            \label{assumption:limited_interaction}
            \normalfont No direct communication, that is, any means of exchange of information, is permitted between agents. 
        \end{assumption}

        \begin{assumption}[Limited Sensing Capacity]
            \label{assumption:limited_sensing}
            Each agent $i\in\mathcal{V}$ only has access to \textit{local} bearing measurements $\bm{\varphi}_{iT}(t)$ and $\bm{\varphi}_{ij}(t)$, and has only one neighbor $j\in\mathcal{N}_i$ where $j=i+1$ if $i<n$ and $j=1$ if $i=n$, for all $t\geq 0$. 
        \end{assumption}

        \begin{remark}
            \label{remark:Graph_Theory}
            Given \textit{Assumption~\ref{assumption:limited_sensing}}, the digraph $\mathcal{G}$ is strongly connected, nonswitching, and balanced.
        \end{remark}

        
        \begin{assumption}[A Priori Knowledge]
             We assume for each agent $i$, the location of other agents and that of the target are not known \textit{a priori}. But it is assumed that each agent $i$ knows \textit{a priori} the desired circumnavigation distance $d^* > 0$ and the desired angular separation $\beta_{ij}^{*}\in [0,2\pi)$. In addition, we assume each agent does not have knowledge about the total number of agents, $n$.
        \end{assumption}

\section{Proposed Algorithm}
\label{sec:Proposed_algorithm}
    Adopting the target estimator proposed in \cite{deghat2010Target,deghat2014Localization}, agent $i$'s estimate of the target location is obtained by
    \begin{equation}
        \label{eq:target_estimator}
        \dot{\hat{\bm{x}}}_i(t) = k_{est} \left( \bm{I}_{2} - \bm{\varphi}_{iT}(t) \bm{\varphi}_{iT}^\top (t) \right) \left( \bm{p}_i(t) - \hat{\bm{x}}_i (t) \right) ,
    \end{equation}
    where $k_{est} \in \mathbb{R}^{+}$ is an estimation gain, and $\bm{I}_{2}$ is a $2\times 2$ identity matrix.

    To distribute agents on the desired circle around the target, we employ an estimate $\hat{\beta}_{ij}(t) \in [-\pi,3\pi)$, which approximates the actual angular separation $\beta_{ij}(t)$ using the following equation,
    \begin{equation}
        \label{eq:angle_estimator}
        \hat{\beta}_{ij}(t) = \begin{cases}
            2\psi_{ij}(t) -3\pi & \text{if } \psi_{ij}(t) \geq \pi , \\
            2\psi_{ij}(t) + \pi & \text{otherwise.}
        \end{cases}
    \end{equation}
    
    A control law should be designed such that all $\hat{\beta}_{ij}$ converge to $\beta_{ij}$ and eventually to the desired values $\beta_{ij}^{*}$. With this goal in consideration, we propose the following controller:
    \begin{equation}
        \label{eq:controller_design}
        \resizebox{.88\hsize}{!}{$\bm{u}_i (t) = k_c \left( \hat{d}_i(t)-d^*\right) \bm{\varphi}_{iT}(t) + k_{\omega}\left( \alpha + \tilde{\beta}_{i}(t) \right) \bar{\bm{\varphi}}_{iT}(t)$},
    \end{equation}
    where $\hat{d}_i(t) = ||\bm{p}_i(t) - \hat{\bm{x}}_i(t)||$, $\tilde{\beta}_{i}(t)=\hat{\beta}_{ij}(t) - \beta_{ij}^{*}$, and $k_c, k_{\omega} > 0$ and $\alpha$ are control constants in which $\alpha$ should be chosen to meet
    \begin{equation}
        \label{eq:alpha_condition}
        \alpha - \sup\limits_{i\in \mathcal{V},t\geq 0} |\tilde{\beta}_{i}(t)| \geq \kappa_{\alpha} > 0,
    \end{equation}
    where $\kappa_{\alpha}$ is some small positive constant.

\section{Stability Proofs}
\label{sec:Stability_proofs}
    In this section, we show that the proposed estimators \eqref{eq:target_estimator}, \eqref{eq:angle_estimator}, and controller \eqref{eq:controller_design} solve \textit{Problem~\ref{problem:Localization_Problem}} and \textit{Problem~\ref{problem:Circum_n_Formation_Problem}}. To this end, we first propose some lemmas.

    \begin{lemma}
        \label{lemma:bounded_estimation_error_x_tilde}
        Under the target estimator \eqref{eq:target_estimator} and the controller \eqref{eq:controller_design}, the estimation error in target's position is bounded by 
        \begin{equation}
            \label{eq:bounded estimation error}
            ||\tilde{\bm{x}}_i (t) || \leq || \tilde{\bm{x}}_i (0) ||, \quad \forall \, i \in \mathcal{V} \wedge t \geq 0.
        \end{equation}
    \end{lemma}

    \begin{proof}
        The proof of \textit{Lemma~\ref{lemma:bounded_estimation_error_x_tilde}} is similar to that of \textit{Lemma~3} in \cite{deghat2010Target} and is therefore omitted.
    \end{proof}
    
    Because the target estimator \eqref{eq:target_estimator} and the controller \eqref{eq:controller_design} make use of the signal $\bm{\varphi}_{iT}(t)$, we should show that $\bm{\varphi}_{iT}(t)$ in \eqref{eq:def_of_varphi_iT} is well-defined for all $t\geq 0$. A sufficient condition is proposed in \textit{Lemma~\ref{lemma:d_i_is_bounded}} below.

    \begin{lemma}
        \label{lemma:d_i_is_bounded}
        Under the target estimator \eqref{eq:target_estimator} and the controller \eqref{eq:controller_design}, for some $d_{i}^{max}>0$ and some $d_{i}^{min}$ satisfying $0<d_{i}^{min}<d^*$, if the initial conditions satisfy
        \begin{align}
            \label{eq:initial_estimate_condition}
            \begin{split}
                d_i(0) \geq d_{i}^{min}&, \quad \forall \, i \in \mathcal{V},\\
                 || \tilde{\bm{x}}_i (0) || = || \hat{\bm{x}}_i (0) - \bm{x}|| \leq &d^{*}-d_{i}^{min}, \quad \forall \, i \in \mathcal{V},
            \end{split}
        \end{align}
        then the agent-target distance $d_i(t)$ is bounded by 
        \begin{equation}
            \label{eq:d_i_is_bounded}
            d_{i}^{min} \leq d_i(t) \leq d_{i}^{max}, \quad \forall \, i\in\mathcal{V} \wedge \forall \, t\geq 0.
        \end{equation}
    \end{lemma}

    \begin{proof}
        \label{proof:d_i_is_bounded}
        We first define two auxiliary variables, 
        \begin{subequations}
            \label{eq:auxiliary_vars}
            \begin{align}
                \delta_i(t) &= d_i (t) - d^*,  \label{eq:delta_i_def} \\
                \rho_i (t) &= d_i (t) - \hat{d}_i (t) , \label{eq:rho_i_def}
            \end{align}
        \end{subequations}
        and let $\bm{\delta}(t) = [\delta_1(t),\dots,\delta_n(t)]^\top$. Using triangular inequality and \eqref{eq:rho_i_def}, we have $ ||\tilde{\bm{x}}_i(t) || \geq |d_i (t) - \hat{d}_i (t) |$, which immediately implies using \textit{Lemma~\ref{lemma:bounded_estimation_error_x_tilde}} that 
        \begin{equation}
            \label{eq:rho_is_bounded}
            |\rho_i (t)| \leq ||\tilde{\bm{x}}_i (t)|| \leq || \tilde{\bm{x}}_i (0) ||.
        \end{equation}
        The dynamics of $ \delta_i (t)$ can be expressed using \eqref{eq:def_of_varphi_iT} and \eqref{eq:controller_design} as (see also \cite[Eq.~(3)]{deghat2010Target}),
        \begin{equation}
            \label{eq:delta_dynamics}
            \dot{\delta}_i (t) = - \delta_i (t) + \rho_i (t) . 
        \end{equation}
        Using \eqref{eq:auxiliary_vars} and \eqref{eq:delta_dynamics}, the agent-target distance $d_i(t)$ can be explicitly found as,
        \begin{equation}
            \label{eq:d_i_solution}
            d_i (t) =  d^* + \delta_i (0) e^{-t} + \int_{0}^{t} e^{-(t-\tau)}\rho_i (\tau) \, d\tau  .
        \end{equation}
        Since $\rho_i(t)$ is bounded by \eqref{eq:rho_is_bounded}, $d_i(t)$ is upper bounded by some $d_{i}^{max}>0$. Further, using \eqref{eq:rho_is_bounded} and \eqref{eq:d_i_solution}, we have,
        \begin{align*}
            d_i(t) \overset{\eqref{eq:rho_is_bounded},\eqref{eq:d_i_solution}}&{\geq} d_i(0) e^{-t} + \left(d^* - ||\tilde{\bm{x}}_i(0)||\right) \left(1-e^{-t}\right) \\
            \overset{\eqref{eq:initial_estimate_condition}}&{\geq} d_{i}^{min} \left(e^{-t} + 1 - e^{-t}\right) = d_{i}^{min} . \numberthis \label{eq:d_i_lower_bound}
        \end{align*}
        Therefore, $d_i(t)$ is lower bounded by $d_{i}^{min}$.
    \end{proof}

    \begin{lemma}
        \label{lemma:varphi_is_persistent_exciting}
        Under the target estimator \eqref{eq:target_estimator} and the controller  \eqref{eq:controller_design}, the signal $\bar{\bm{\varphi}}_{iT}(t)$ is persistently exciting (p.e.) for all $i\in \mathcal{V}$ and $\forall \, t > 0$.
    \end{lemma}

    \begin{proof}
        \label{proof:varphi_is_persistent_exciting}
        According to \cite{sastry1994Adaptive}, the signal $\bar{\bm{\varphi}}_{iT}(t)$ is considered p.e. if there exist positive constants $\sigma_1$, $\sigma_2$, $T$ such that
        \begin{equation}
            \label{eq:scalar form of p.e. condition}
            \sigma_1 \leq \int_{t_0}^{t_0+T} \left( \bm{U}_{i}^{\top} \bar{\bm{\varphi}}_{iT}(t) \right)^2 \, dt \, \leq \sigma_2 
        \end{equation}
        holds for all constant unit vector $\bm{U}_i \in \mathbb{R}^2$ (see Fig.~\ref{fig:notation_conventions}) and any positive constant $t_0$. Equation \eqref{eq:scalar form of p.e. condition} can be expressed in terms of the angle $\gamma_i (t)$ as
        \begin{equation}
            \label{eq:p.e. condition (cos form)}
            \sigma_1 \leq \int_{t_0}^{t_0 + T} \cos^2 \gamma_i (t) \, dt\, \leq \sigma_2 ,
        \end{equation}
        where $\gamma_i (t)$ is the angle from the unit vector $\bm{U}_i$ to the unit vector $\bm{\varphi}_{iT}(t)$, as depicted in Fig.~\ref{fig:notation_conventions}. Noticing that $\bar{\bm{\varphi}}_{iT}(t) \perp \bm{\varphi}_{iT}(t)$ and the angle $\xi_i(t) - \gamma_i(t)$ is always constant, we have
        \begin{equation}
            \label{eq:pe_condition_inter_1}
            \frac{d\gamma_i(t)}{dt} = \frac{d\xi_i (t)}{dt} = \frac{d\beta_i (t)}{dt},
        \end{equation}
        where $\beta_i(t)$ is the angle rotated counterclockwise from the $x$-axis of the $i$'s local frame to the unit vector $\bm{\varphi}_{iT}(t)$, as illustrated in Fig.~\ref{fig:notation_conventions}. For the purpose of stability analysis \textit{only}, we assume that all agents maintain a local frame aligned to the global frame with the origin fixed at $\bm{x}$. Note that the controller \eqref{eq:controller_design} does not use any term involving $\beta_{i}(t)$. Using \eqref{eq:pe_condition_inter_1} and \eqref{eq:controller_design}, we obtain
        \begin{equation}
            \frac{d\gamma_i (t)}{dt} = \frac{d\beta_i (t)}{dt} =  \medmath{ \frac{ k_{\omega}\left(\alpha + \tilde{\beta}_{i}(t) \right)}{d_i(t)} } \overset{\eqref{eq:alpha_condition},\eqref{eq:d_i_is_bounded}}{\geq} \medmath{\frac{k_\omega \kappa_{\alpha} }{d_{i}^{max}}}.
        \end{equation}
        Consequently, it holds for all $t\geq 0$ that
        \begin{equation}
            \gamma_i (t + t_0) \geq \gamma_i (t_0) + \medmath{\frac{k_\omega \kappa_{\alpha} t} {d_{i}^{max}}} ,
        \end{equation}
        which means $\gamma_i (t)$ is monotonically increasing, and it follows that there always exist $\sigma_1 , T >0$ that satisfy \eqref{eq:p.e. condition (cos form)}.
    \end{proof}

    \begin{lemma}
        \label{lemma:tilde_beta_ij_is_bounded}
        The angular separation error $\tilde{b}_{i}(t)=\beta_{ij}(t) - \beta_{ij}^{*}$ is bounded for all $i\in \mathcal{V}$ and $t\geq 0$.
    \end{lemma}

    \begin{proof}
        Introduce the following auxiliary variable,
        \begin{equation}
            \label{eq:angular_estimation_error} 
            \tilde{\vartheta}_{i}(t) = \beta_{ij}(t) - \hat{\beta}_{ij}(t) ,
        \end{equation}
        and the auxiliary functions,
        \begin{align}
            q_i(t) &= \medmath{\frac{k_\omega}{d^*} \left(\tilde{b}_j(t) - \tilde{b}_{i}(t)\right)} ,\label{eq:nominal_system_q_i}\\
            g_i(t) &= \medmath{\alpha k_\omega  \left(\frac{1}{\delta_j(t)+d^*} - \frac{1}{\delta_i(t) + d^*}\right) -\frac{k_\omega}{d^*} \left(\tilde{b}_{j}(t)-\tilde{b}_{i}(t)\right)} \nonumber \\
            &\quad \medmath{+ k_\omega  \left( \frac{\tilde{b}_j(t) - \tilde{\vartheta}_j(t)}{\delta_j(t)+d^*} - \frac{\tilde{b}_{i}(t)-\tilde{\vartheta}_{i}(t)}{\delta_i(t)+d^*} \right)} , \label{eq:perturbation_g_i} 
        \end{align}
        and define the vectors $\bm{q}(t)$ and $\bm{g}(t)$ as $\bm{q}(t) = [q_1(t),\dots,q_n(t)]^\top$ and $\bm{g}(t) = [g_1(t),\dots,g_n(t)]^\top$.
        
        The angular separation $\beta_{ij}(t)$ can be written in terms of $\beta_i(t)$, as follows (see also Fig.~\ref{fig:notation_conventions})
        \begin{equation}
            \label{eq:beta_ij_using_beta_i}
            \resizebox{.8\hsize}{!}{$
            \beta_{ij}(t) = \begin{cases}
                \beta_j (t) - \beta_i (t)  & \text{if } \beta_j (t) > \beta_i(t) , \\
                \beta_j (t) - \beta_i (t) + 2\pi  & \text{otherwise.}
            \end{cases}
            $}
        \end{equation}
        Noticing that $\tilde{\beta}_i(t) = \tilde{b}_{i}(t) - \tilde{\vartheta}_{i}(t)$, the time derivative of $\tilde{b}_{i}(t)$ can be expressed as follows,
        \begin{align}
            \dot{\tilde{b}}_i(t) &= \dot{\beta}_{ij}(t) - \dot{\beta}_{ij}^{*} \nonumber  \\
            \overset{\eqref{eq:beta_ij_using_beta_i}}&{=}  \dot{\beta}_j(t)- \dot{\beta}_i(t) - 0 \nonumber \\
            \overset{\eqref{eq:controller_design},\eqref{eq:delta_i_def}}&{=}  \medmath{ \frac{k_\omega\left(\alpha+\tilde{\beta}_{j}(t)\right)}{\delta_j(t)+d^*} - \frac{k_\omega \left(\alpha + \tilde{\beta}_i(t)\right)}{\delta_i(t)+d^*} } \nonumber \\
            \overset{\eqref{eq:nominal_system_q_i},\eqref{eq:perturbation_g_i}}&{=} q_i(t) + g_i(t). \label{eq:b_tilde_i_dot}
        \end{align}
        By defining $\tilde{\bm{b}}(t) = [ \tilde{b}_1(t), \dots , \tilde{b}_n(t) ]^\top$, the error dynamics \eqref{eq:b_tilde_i_dot} can be written in a vector form as a perturbed system, 
        \begin{align}
            \Pi \; : \; \dot{\tilde{\bm{b}}}(t) &= \bm{q}(t), \label{eq:b_tilde_i_nominal} \\
            \Pi' \; : \; \dot{\tilde{\bm{b}}}(t) &= \bm{q}(t) + \bm{g}(t) .\label{eq:b_tilde_i_as_perturbed_sys} 
        \end{align}
        in which \eqref{eq:b_tilde_i_nominal} is the nominal system and \eqref{eq:b_tilde_i_as_perturbed_sys} is the perturbed system with $\bm{g}$ being the perturbation term.
        
        Now consider each condition in \textit{Lemma~\ref{lemma:perturbed_sys}}:
        
        \underline{(C1) of \textit{Lemma~\ref{lemma:perturbed_sys}}:} From \textit{Lemma~\ref{lemma:d_i_is_bounded}}, we know that $d_i(t)$ is lower bounded. Noticing also that $\tilde{b}_{i}(t) \in (-2\pi,2\pi)$ and $\tilde{\vartheta}_i(t) \in (-3\pi,3\pi)$ for all $i\in\mathcal{V}$ and $t\geq 0$, we can bound the perturbation term $\bm{g}(t)$, as follows,
        \begin{equation}
            \label{eq:bounded_perturbation_varpi}
            ||\bm{g}(t)|| < \medmath{\frac{\sqrt{n} \, k_\omega \left(\alpha + 14 \pi \right)}{\min\limits_{i\in\mathcal{V}}d_{i}^{min}} } =: \varpi .
        \end{equation}

        \underline{(C2) of \textit{Lemma~\ref{lemma:perturbed_sys}}:} 
        Rewrite the nominal system \eqref{eq:b_tilde_i_nominal} in terms of $\tilde{\bm{b}}$, we obtain
        \begin{equation}
            \label{eq:B_tilde_dot}
            \dot{\tilde{\bm{b}}} = -\medmath{\frac{k_\omega}{d^*}\mathcal{L}(\mathcal{G})\tilde{\bm{b}} ,}
        \end{equation}
        where $\mathcal{L}(\mathcal{G})$ is the Laplacian matrix. Now consider the Lyapunov candidate $V(\tilde{\bm{b}}) = \frac{1}{2}\tilde{\bm{b}}^\top \tilde{\bm{b}}$ and its derivative along the trajectories of the nominal system \eqref{eq:B_tilde_dot}, 
        \begin{align*}
            \dot{V} &= \tilde{\bm{b}}^\top \dot{\tilde{\bm{b}}} 
            \overset{\eqref{eq:B_tilde_dot}}{=} - \medmath{\frac{k_\omega}{d^*}} \tilde{\bm{b}}^\top \mathcal{L}(\mathcal{G})\tilde{\bm{b}}
            \overset{\textit{Lemma~\ref{lemma:algebraic_connectivity}}}{\leq} - \medmath{\frac{k_\omega \lambda_2(\mathcal{G})}{d^*} }\tilde{\bm{b}}^\top  \tilde{\bm{b}} ,
        \end{align*}
        where $\lambda_2(\mathcal{G})$ is the algebraic connectivity. Note that $\bm{1}_n^\top \tilde{\bm{b}}=\sum \beta_{ij}(t) - \sum \beta_{ij}^* = 2\pi - 2\pi = 0$. Therefore, we conclude that the origin $\tilde{\bm{b}}=\bm{0}$ of the nominal system \eqref{eq:b_tilde_i_nominal} is globally exponentially stable (GES). 

        \underline{(C3) of \textit{Lemma~\ref{lemma:perturbed_sys}}:} By choosing $W_1(\tilde{\bm{b}}) = W_2(\tilde{\bm{b}}) = \frac{1}{2} \tilde{\bm{b}}^\top \tilde{\bm{b}}$ and $W_3(\tilde{\bm{b}}) = \frac{c_3}{2} \tilde{\bm{b}}^\top \tilde{\bm{b}}$ where $c_3 = \frac{k_\omega \lambda_2(G)}{d^*}$, it is easy to show that the origin of the nominal system \eqref{eq:B_tilde_dot} is GES, and therefore GAS.

        Therefore, we conclude that the solutions to the perturbed system \eqref{eq:b_tilde_i_as_perturbed_sys} are uniformly bounded.
    \end{proof}

    We now present the main results of this paper.

    \begin{theorem}
        \label{theorem:estimation_error_x_tilde_goes_to_zero}
        Under the target estimator \eqref{eq:target_estimator} and the controller \eqref{eq:controller_design}, the estimation error $\tilde{\bm{x}}_i(t)$ converges to zero exponentially fast for all $i\in\mathcal{V}$.
    \end{theorem}
    
    \begin{proof}
        \label{proof:estimation_error_x_tilde_goes_to_zero}
        The proof is an immediate consequence of \textit{Lemma~\ref{lemma:varphi_is_persistent_exciting}} and \textit{Theorem~1} of \cite{anderson1997Exponential}.
    \end{proof}
    
    \begin{theorem}
        \label{theorem:tracking_error_converges_to_zero}
        Under the target estimator \eqref{eq:target_estimator} and the controller \eqref{eq:controller_design}, the tracking error $\delta_{i}(t)$ converges to zero exponentially fast.
    \end{theorem}
    
    \begin{proof}
        \label{proof:tracking_error_converges_to_zero}
        From \textit{Theorem~\ref{theorem:estimation_error_x_tilde_goes_to_zero}}, we know that $\tilde{\bm{x}}_i(t)$ converges to zero exponentially fast, and because of \eqref{eq:rho_is_bounded}, the signal $\rho_i(t)$ also converges to zero exponentially fast. Then from \eqref{eq:delta_dynamics}, the tracking error $ \delta_i(t)$ converges to zero exponentially fast.
    \end{proof}

    \begin{theorem}
        \label{theorem:formation}
        Under the estimator \eqref{eq:target_estimator} and \eqref{eq:angle_estimator}, and the controller \eqref{eq:controller_design}, the angular separation $\beta_{ij}(t)$ asymptotically converges to the desired separation $\beta_{ij}^{*}$.
    \end{theorem}

    \begin{proof}
        The error dynamics \eqref{eq:b_tilde_i_as_perturbed_sys} and \eqref{eq:delta_dynamics} can be seen as a cascaded system, as follows,
        \begin{subequations}
            \label{eq:formation_cascade}
            \begin{align}
                \sbox{\mycases}{$\displaystyle \Sigma\left\{\begin{array}{@{}c@{}}\vphantom{1\ x\geq0}\\\vphantom{0\ x<0}\end{array}\right.\kern-\nulldelimiterspace$}
              \raisebox{-.5\ht\mycases}[0pt][0pt]{\usebox{\mycases}}\Sigma_1 \; : \; \dot{\tilde{\bm{b}}}(t) &= \bm{q}(t) + \bm{g}(t) =: \bm{f}_1 (\tilde{\bm{b}}(t),\bm{\delta}(t)),\label{eq:cascade_system_1_b_tilde} \\
                 \Sigma_2 \; : \; \dot{\bm{\delta}}(t) &=  \bm{f}_2(\bm{\delta}(t)),  \label{eq:cascade_system_2_delta}
            \end{align}
        \end{subequations}
        where $\bm{f}_2(\bm{\delta}(t)):=[-\delta_1(t) + \rho_1(t),\dots,-\delta_n(t)+\rho_n(t)]^\top$.
        
        We now consider each condition in \textit{Lemma~\ref{lemma:cascaded system}}:
        
        \underline{(C1) of \textit{Lemma~\ref{lemma:cascaded system}}:} If $\bm{\delta}(\tau) =\bm{0}$ at $t=\tau$, then all agents are on the desired circle around the target which means every triangle formed by vertices $\bm{x}$, $\bm{p}_i(\tau)$, and $\bm{p}_j(\tau)$ is isosceles. Therefore, by simple trigonometry, we have $\vartheta_i(\tau) = \beta_{ij}(\tau) - \hat{\beta}_{ij}(\tau) = 0$ for all $i\in\mathcal{V}$. It then follows that
        \begin{equation}
            \label{eq:cascaded_f1(x1,0)}
            \bm{f}_1(\tilde{\bm{b}}(t),\bm{0}) = \bm{q}(t) + \bm{g}(t)|_{\bm{\delta}=\bm{0}} = \bm{q}(t) .
        \end{equation}
        From the proof of \textit{Lemma~\ref{lemma:tilde_beta_ij_is_bounded}}, we know that the origin of $\dot{\tilde{\bm{b}}}(t) = \bm{q}(t)$ is GES. Then, using \eqref{eq:cascaded_f1(x1,0)}, we conclude that the origin of $\bm{f}_1(\tilde{\bm{b}}(t),\bm{0})$ is GES, and therefore also GAS.

        \underline{(C2) of \textit{Lemma~\ref{lemma:cascaded system}}:} From \textit{Theorem~\ref{theorem:tracking_error_converges_to_zero}}, we know that the origin of the subsystem \eqref{eq:cascade_system_2_delta} is GES.

        \underline{(C3) of \textit{Lemma~\ref{lemma:cascaded system}}:} From \textit{Lemma~\ref{lemma:tilde_beta_ij_is_bounded}}, we know that the solutions of the cascaded subsystem \eqref{eq:cascade_system_1_b_tilde} exist and are uniformly bounded. 

        Thus, we conclude that the origin $(\tilde{\bm{b}}(t),\bm{\delta}(t)) = (\bm{0},\bm{0})$ of the cascaded system \eqref{eq:formation_cascade} is GAS.
    \end{proof}

\section{Simulation Results}
\label{sec:Simulation_results}

    We consider a stationary target situated at $\bm{x}(t)\equiv[0,0]^\top$, which is to be localized and circumnavigated by a group of five agents. The prescribed parameters are chosen as $d^{*}=1.2m$ and $\bm{\beta}^{*}=[\frac{5\pi}{18},\frac{\pi}{9},\frac{5\pi}{18},\frac{5\pi}{18},\frac{19\pi}{18}]^{\top}$. The control gains are selected as $k_{est}=5$, $k_c=1.5$, $k_{\omega} = 1$, and $\alpha_1 = 3.5\pi$ according to \eqref{eq:alpha_condition}. The neighbor topology is defined in accordance with \textit{Assumption~\ref{assumption:limited_sensing}}. The trajectories of five agents are depicted in Fig.~\ref{fig:simulation_1}(a). As expected from \textit{Theorem~\ref{theorem:estimation_error_x_tilde_goes_to_zero}}, the norm of the estimation error $||\tilde{\bm{x}}_i(t)||$ converges to zero exponentially fast for all $i\in\mathcal{V}$. According to \textit{Theorem~\ref{theorem:tracking_error_converges_to_zero}} and \textit{Theorem~\ref{theorem:formation}}, we expect that the agent-target distance $d_i(t)$ converges exponentially fast to $d^*$ and the separation angle $\beta_{ij}(t)$ converges to $\beta_{ij}^{*}$, which are confirmed by Fig.~\ref{fig:simulation_1}(c) and Fig.~\ref{fig:simulation_1}(d), respectively. Comparing to controllers using communication-based $\beta_{ij}(t)$ such as \cite{swartling2014Collective,boccia2017Tracking,chen2022Multicircular}, the proposed algorithm acquired the property of not relying on communication at the expense of a slightly slower convergence rate in the angular separation error $\tilde{b}_i(t)$.

    \begin{figure*}[htp]
        \centering
        \includegraphics[width=1.0\linewidth]{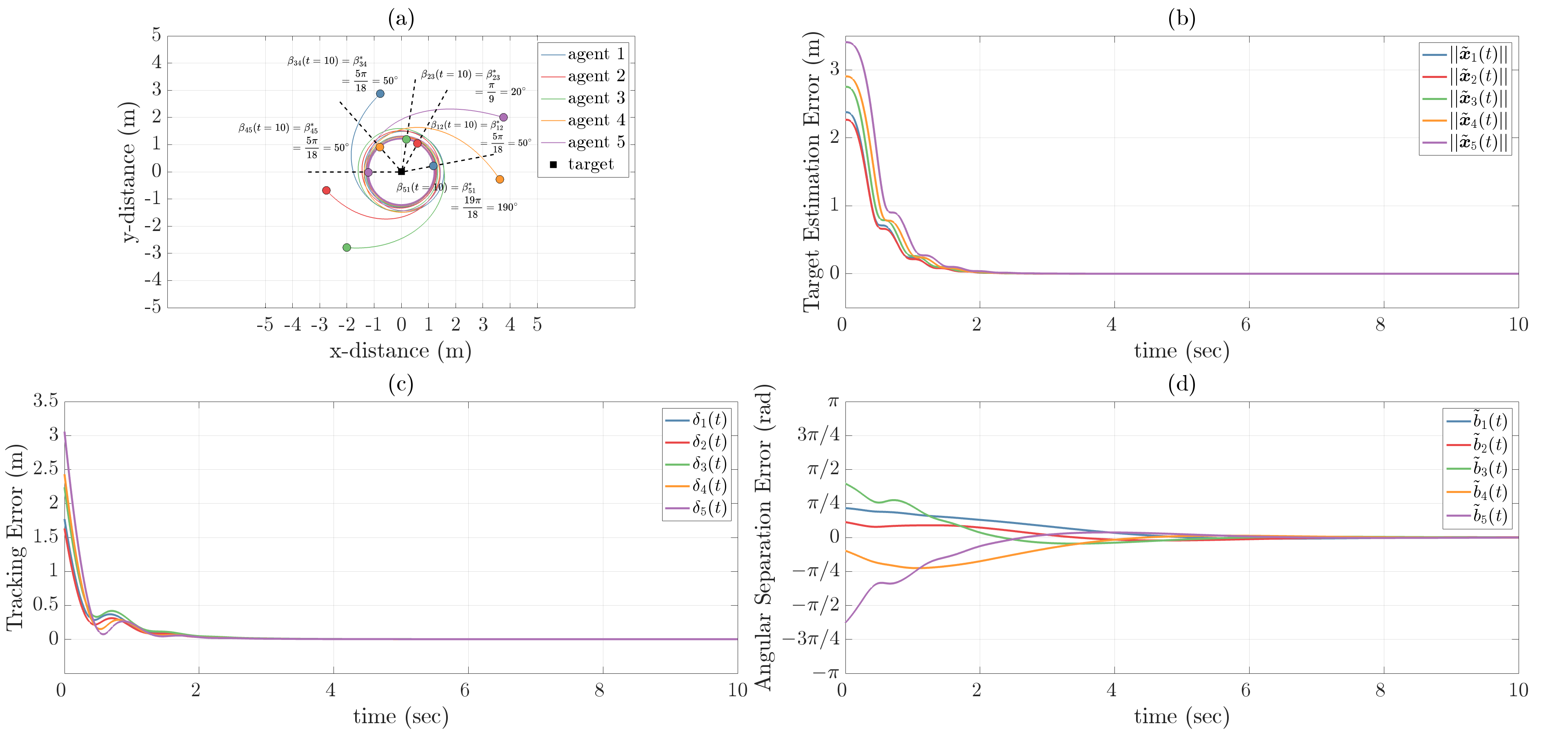}
        \caption{\label{fig:simulation_1}(a) trajectories of agents, (b) \textmd{$||\tilde{\bm{x}}_i(t)||$}, (c) $\delta_i(t)$, (d) $\tilde{b}_{i}(t)$.}
    \end{figure*}

\section{Concluding Remarks and future work}
\label{sec:Conclusion}
    In this paper, we proposed a controller for the BoTLC problem and an algorithm for angular separation. The controller does not require any communication between agents to achieve the prescribed spaced formation. The stability of the proposed algorithm has been analyzed rigorously. Simulation results put in evidence a satisfactory performance of the proposed control algorithms.
    
    Future directions of research include addressing collision avoidance among agents and considering a moving target.


\addcontentsline{toc}{chapter}{Bibliography}
\bibliographystyle{ieeetr}
\bibliography{root_v7}

\end{document}